\documentclass[11 pts]{amsart}

\usepackage{amsfonts, amssymb, amsmath, latexsym, enumerate, array, amsthm}
\usepackage[all]{xy}
\usepackage{cite,mdwlist,mathrsfs}

\newtheorem{thm}{Theorem}[section] \newtheorem{lemma}[thm]{Lemma}
\newtheorem{theorem}{Theorem}[section]

\newtheorem*{thm*}{Theorem}

\theoremstyle{definition} 
 \newtheorem{definition}[thm]{Definition}
 
\newtheorem*{remark}{Remark} \newtheorem*{acknowledgements}{Acknowledgments}
\newtheorem{problem}{Problem}

\newtheorem*{theorem*}{Theorem}
\newtheorem*{remark*}{Remark}
\newtheorem{my-claim}{Claim}

\newcommand{\OO}{\mathscr{O}}

\newcommand{\EExt}{\mathscr{E}xt}

\newcommand{\sL}{\mathscr{L}}

\newcommand{\sM}{\mathscr{M}}

\newcommand{\cH}{\mathcal{H}}
\newcommand{\cE}{\mathcal{E}}

\newcommand{\cL}{\mathcal{L}}
\newcommand{\cF}{\mathcal{F}}

\newcommand{\cT}{\mathcal{T}}
\newcommand{\cN}{\mathcal{N}}

\newcommand{\cS}{\mathcal{S}}

\newcommand{\cI}{\mathcal{I}}
\newcommand{\cP}{\mathcal{P}}

\DeclareMathOperator{\id}{{\sf id}}
\DeclareMathOperator{\GL}{{\sf GL}}

\DeclareMathOperator{\rD}{D}

\DeclareMathOperator{\rk}{rk}

\DeclareMathOperator{\Ext}{Ext} 
\DeclareMathOperator{\Sym}{Sym}
\DeclareMathOperator{\Hom}{Hom} 
\DeclareMathOperator{\im}{Im} \DeclareMathOperator{\cok}{Cok}

\DeclareMathOperator{\HH}{H} \DeclareMathOperator{\hh}{h}

\DeclareMathOperator{\Pic}{Pic}

\DeclareMathOperator{\Pf}{Pf}
\newcommand{\tra}{{}^{\mathrm{t}}}

\newcommand{\Z}{\mathbb Z} 
 
 \newcommand{\p}{\mathbf P}
\newcommand{\ppp}{\mathbb P}
 \newcommand{\G}{\mathbf G}

\DeclareMathOperator{\ts}{\otimes}

\newcommand{\xr}{\xrightarrow}
\newcommand{\kk}{\mathbf{k}}

\SelectTips{cm}{} \newdir{ >}{{}*!/-5pt/@{>}}
\numberwithin{equation}{section}
\allowdisplaybreaks[1]
\begin{document}
%\linenumbers
%\linenumberdisplaymath

%%%%%%%%%%%%%%%%%%%%%%%%%%%%%%%%%%%%%%%%%%%%%%%%%%%%%%%%%%%%%%%%%%%%%%%%%%%%%%
%%%%%%%%%%%%%%%%%%%% Author(s) and Address %%%%%%%%%%%%%%%%%%%%%%%%%%%%%%%%%%%
%%%%%%%%%%%%%%%%%%%%%%%%%%%%%%%%%%%%%%%%%%%%%%%%%%%%%%%%%%%%%%%%%%%%%%%%%%%%%%

\title{Skew-symmetric matrices and Palatini scrolls}
\thanks{The first named author was partially supported by {\sc GRIFGA},
  MIUR, and ANR contract Interlow ANR-09-JCJC-0097-0. The second named author was partially supported by MIUR funds.}

\author{Daniele Faenzi}
\address{Universit\'e de Pau et des Pays de l'Adour, Avenue de l'Universit\'e - BP 576 - 64012 PAU Cedex - France,}
\email{daniele.faenzi@univ-pau.fr}
\author{Maria Lucia Fania}
\address{Universit\`{a} degli Studi dell'Aquila,
Via Vetoio Loc. Coppito\\ 67100 L'Aquila, Italy,} \email{fania@univaq.it}

% \date{Received: date / Accepted: date}
% The correct dates will be entered by the editor



\maketitle

\begin{abstract}
 We prove that, for $m$ greater than $3$ and $k$ greater than $m-2$,
 the Grassmannian of $m$-dimensional subspaces of the space of skew-symmetric forms over a vector 
 space of dimension $2k$ is birational to the Hilbert scheme of Palatini scrolls in $\p^{2k-1}$.
 
 For $m=3$ and $k > 3$, this Grassmannian 
 is proved to be birational to the set of pairs $(\cE,Y)$, where $Y$ is a smooth
 plane curve of degree $k$ and $\cE$ is a stable rank-$2$ bundle on $Y$
 whose determinant is $\OO_Y(k-1)$.
\end{abstract}

%\maketitle
%\today

\section{Introduction}

Let $\kk$ be a field, $n \geq 3$ be an integer, and $V$ be a
$\kk$-vector space of dimension $n+1$.
Degeneracy loci of general morphisms $\phi:
\OO^{m}_{\p(V)} \to \Omega_{\p(V)}(2)$, have been considered by
several authors, for example M. C. Chang \cite{chang} and G. Ottaviani
\cite{ottaviani:scrolls}. 
Relying on a nice interpretation due to Ottaviani, we identify the
global sections of $\Omega_{\p(V)}(2)$ with $\wedge^{2} V^{*}$, the
space of skew-symmetric forms on $V$, or of skew-symmetric matrices of
size $n+1$ with
coefficients in the base field $\kk$. 
Thus a morphism $\phi: \OO^{m}_{\p(V)} \to \Omega_{\p(V)}(2)$ can be
written in coordinates by means of $m$ skew-symmetric
matrices $A_1, \ldots, A_m$.

In classical terminology, each matrix $A_i$ corresponds to a 
linear line complex $\Gamma_{i}$ of $\p(V)$, so that the degeneracy locus $X_\phi$
of $\phi$, can be thought of as the set of centers of complexes belonging
to the linear system spanned by $\Gamma_{1}, \ldots, \Gamma_{m}$.
Under this point of view, these varieties were considered already by classical
algebraic geometers.
For instance, in $1891$ G. Castelnuovo in \cite{castelnuovo} considered the
case $m=3$ and $n=4$, namely the case of nets of linear
complexes in $\p^4$. The locus of lines which are centers of the linear
complexes belonging to a general net of linear complexes in $\p^4$,
or, in modern language, the degeneracy locus of a general morphism
$\phi: \OO^{3}_{\p^4} \to \Omega_{\p^4}(2)$, is the projected Veronese
surface in $\p^4$.

In $1901$ F. Palatini in \cite{palatini:cinque} and \cite{palatini:complessi} considered the case of linear systems of linear
complexes in $\p^5$ of dimension $> 1$.  The case $m=3$ leads to the
elliptic scroll surface of degree $6$, which was further studied by
G. Fano \cite{fano:complessi-lineari}. The case $m=4$ gives a $3$-fold
of degree $7$ which is a scroll over a cubic surface of $\p^3$, also
called Palatini scroll.  
%further studied by G. Fano in \cite{fano} and 
%reviseted by Bazan, Mezzetti in \cite{bazan-mezzetti}.

In the present paper we will consider the Hilbert scheme $\cH_m(V)$ of
these degeneracy loci.
Given a general map $\phi: \OO^{m}_{\p(V)} \to \Omega_{\p(V)}(2)$, 
we consider $\phi$ as the point $[\phi]$ of the Grassmann variety
$\G(m,\wedge^{2} V^{*})$ parametrizing $m$-dimensional vector
subspaces of $\wedge^{2} V^{*}$.
Since a linear change of coordinate on $U$ does not change the degeneracy locus $X_\phi$,
we have a rational map:
\begin{align*} \rho : \G(m,\wedge^{2} V^{*}) & \dashrightarrow
  \cH_m(V) \\
  [\phi] & \mapsto X_{\phi},
\end{align*}
where $\cH_m(V)$ denotes the union of components of the Hilbert scheme
of $\p(V)$ containing the degeneracy locus of a general map of the
form $\phi$.

As an instance of classical results in this direction, 
let us mention that, if $m=3$ and $n=4$, 
from the results contained in \cite{castelnuovo}, one can prove that
the component of $\cH_m(V)$ containing 
Veronese surfaces in $\p^4$ is birational to $\G(3,\wedge^{2} V^{*})$.
A similar statement holds for the Palatini scroll in $\p^5$:
the main result of \cite{fania-mezzetti} states that $\rho$ is a birational map if
$m=4$, $n=5$.
On the other hand, it was proved in \cite{bazan-mezzetti}, and in fact classically known
to Fano (see \cite{fano:complessi-lineari}), that the map $\rho$ is
generically $4:1$ in case $m=3$, $n=5$.

We focus here on the case when $n$ is odd (say $n=2k-1$), in
which case we call a scroll of the form $X_\phi$ a {\it Palatini scroll}, including also the case of surface scrolls over a curve.
This last case is particularly well-studied, in the framework of surface
scrolls in $\p(V)$ which are non-special.
In particular we mention the results on the 
Hilbert schemes of non-special scrolls due to
\cite{calabri-ciliberto-flamini-miranda:non-special}.

In the present paper we
generalise the result of \cite{fania-mezzetti}
by proving that it holds for all $m\geq 4$ and all $k\geq m-1$.
Moreover, we prove that, for $m=3$, the map $\rho$ is generically injective for
all $k\geq 4$. In other words, we show that the case studied in
\cite{bazan-mezzetti} is the only exception to injectivity of $\rho$.
More precisely, we prove:

\begin{theorem*} Let $m,k$ be integers and let $\rho$ be the map
introduced above.
  \begin{enumerate}[A)]
  \item The map $\rho$ is birational for all $m\geq 4$, $k\geq
    m-1$, in particular the Hilbert scheme $\cH_m(V)$ is generically
    smooth of dimension $m(k(2k-1)-m)$.
  \item For $m=3$ and for all $k\geq 4$, the map $\rho$ is
    generically injective. Moreover, it is dominant on the closed subscheme of
    $\cH_3(V)$ whose general element is a general plane curve $Y$ of degree
    $k$ equipped with a general stable rank-$2$ bundle whose determinant
    is $\OO_Y(k-1)$.
  \end{enumerate}
\end{theorem*}

\begin{remark*}
  For $m\geq 4$ the general element $X$ of $\cH_m(V)$ is a smooth
  subvariety $X$ of $\p(V)$, whose normal bundle $\cN$ satisfies
  $\dim \HH^0(X,\cN)=m(k(2k-1)-m)$. For $m = 4$, $k\geq 5$, we have
  $\HH^1(X,N)\neq 0$, nevertheless this space gives no obstructions to
  $\cH_m(V)$.

  For $m=3$, the general element of $\cH_m(V)$ is a scroll of the form
  $\ppp(\cE)$, where $\cE$ is a stable bundle of rank $2$ and degree
  $k(k-1)$ over a smooth curve of genus ${k-1 \choose 2}$.
  In other words, the curve is not planar in general, and  only the
  degree of $\cE$ is fixed, not its determinant.
\end{remark*}

%%% \begin{remark} The theorem says that $\G(m,\wedge^{2} V^{*})$ is
%%%   birational to the moduli space $\cP$ of embedded Palatini scrolls. In
%%%   turn, $\cP$ is birational to $\cH_m(V)$ for $m\geq 4$, while for $m=3$ it
%%%   can be described as the set of pairs $(\cE,Y)$, where $Y$ is a smooth
%%%   plane curve of degree $k$ and $\cE$ is a stable rank-$2$ bundle on $Y$
%%%   with $c_{1}(\cE(-H_{Y}))=K_{Y}$.
%%% \end{remark}

%%%%%%%%%%%%%%%%%%%%%%%%%%%%%%%%%%%%%%%%%%%%%%%%%%%%

\section{Basic constructions}

In this section we review some classical constructions related to
Palatini scrolls. All the material in this section is essentially well-known.

%%%%%%%%%%%%%%%%%%%%%
%\subsection{Preliminaries and notation}
Let $\kk$ be a field.
Given a $\kk$-vector space $A$ of finite dimension, we set $\p(A)$ for
the projective space of
$1$-dimensional subspaces of $A$. We use the notation $\G(a,A)$ for
the Grassmannian of $a$-dimensional vector subspaces of $A$. 
Given a torsion-free coherent sheaf  $\cE$ on an integral scheme $Y$,
let $\cS=\Sym(\cE)$ be the symmetric algebra of $\cE$.  With $\ppp(\cE)$,
the associated  projective space bundle,  we mean $\ppp(\cE):=\mathrm{Proj}(\cS)$
and we let  $p: \ppp(\cE) \to Y$ denote the  projection morphism. Note that there is a natural map $\ppp(\cE) \to \p(\HH^0(Y,\cE)^{*}).$

%%%%%%%%%%%%%%%%%%%%%%%%
\subsection{Skew-symmetric forms and scrolls}

Let $m$ and $k$ be two integers, with $2k\geq m+1$, and let $V$ be a $\kk$-vector space of
dimension $2k$ and $U$ be a $\kk$-vector space of dimension $m$.
Let us introduce the projective spaces $\p(V)$ and $\p(U)$.
We consider the space $\wedge ^{2} V^{*}$ of skew-symmetric linear forms on
$V$. This is well-known to be canonically isomorphic to the space
$\HH^{0}(\p(V),\Omega_{\p(V)}(2))$ of twisted $1$-forms on $\p(V)$.

Let $\phi$ be an injective map $U \to \wedge^{2} V^{*}$.
We denote by $[\phi]$ the point of the Grassmannian
$\G(m,\wedge^{2} V^{*})$ of $m$-dimensional vector subspaces of
$\wedge^{2} V^{*}$ corresponding to $\phi$.
In view of the identification
$\HH^{0}(\p(V),\Omega_{\p(V)}(2)) \cong \wedge^{2} V^{*}$,
we may consider $\phi$ as a map:
\begin{equation}
  \label{phi}
  U \ts \OO_{\p(V)} \to \Omega_{\p(V)}(2),  
\end{equation}
which we still denote by $\phi$.

\begin{definition}
  Given a map $\phi : U \to \wedge^{2} V^{*}$ as in \eqref{phi},
  we define the subscheme $X_\phi \subset \p(V)$ as the degeneracy
  locus $\rD_{m-1}(\phi)$ of $\phi$.
  The underlying set of $X_\phi$ consists of the points of $\p(V)$ where $\phi$ is not of maximal rank.
  If the codimension of $X=X_\phi$ is $2k-m$, we say that $X_\phi$ is the
  {\it Palatini scroll} in $\p(V)$ associated to $\phi$.
  In turn, the Hilbert scheme $\cH_m(V)$ of Palatini scrolls is defined as the
  (union of) component(s) of the Hilbert scheme of subschemes of
  $\p(V)$ containing Palatini scrolls.
%  The {\em moduli space of embedded Palatini scrolls} $\cP_m(V)$ is defined
%  as the closure in $\cH_m(V)$ of the set of scrolls of the form $X_{\phi}$ for some
%  $\phi : U \ts \OO_{\p(V)} \to \Omega_{\p(V)}(2)$, with $\dim_\kk U = m$.
\end{definition}

We refer for instance to \cite{pragacz:enumerative} for basic facts
on degeneracy loci.
Observe that the variety $X$ can be regarded as the degeneracy locus of
the transpose map:
\begin{equation}
 \label{phitrasposta}
\tra \phi: \cT_{\p(V)}(-2) \to U^* \ts \OO_{\p(V)}.
\end{equation}
%%% Let us now assume that the scheme {\it $X_\phi$ is smooth}.
%%% Note that this forces the degeneracy locus $\rD_{m-2}(\phi)$ of $\phi$ to be
%%% empty, so an easy dimension count implies:
%%% \[
%%%   m \leq k+1.
%%% \]
Notice that,
composing $\tra \phi$ with the (twisted) Euler exact sequence on $\p(V)$:
\[
  0 \to \OO_{\p(V)}(-2) \to V \ts \OO_{\p(V)}(-1) \to \cT_{\p(V)}(-2) \to 0,
\]
we obtain a map:
\[
F_{\phi} : V \ts \OO_{\p(V)}(-1) \to U ^* \ts \OO_{\p(V)},
\]
whose image is just $\im(\tra \phi)$.
Thus, a Palatini scroll $X$ defined as the degeneracy locus of
a map $\phi$ can be seen as well as the degeneracy locus of the map $F_{\phi}$.
This map is an $m\times 2k$ matrix of linear forms,
and $X$ is defined by the vanishing of all its $m\times m$ minors.
In coordinates, the map $F_\phi$ looks like:
\[
F_{\phi} = 
\begin{pmatrix}
\sum_{j=1}^{2k} a_{1,j}^1 x_j & \cdots & \sum_{j=1}^{2k} a_{2k,j}^1 x_j\\
\vdots && \vdots \\
\sum_{j=1}^{2k} a_{1,j}^m x_j & \cdots & \sum_{j=1}^{2k} a_{2k,j}^m x_j \\
\end{pmatrix},
\]
for some constants $a_{i,j}^\ell$ satisfying:
\[
a_{i,j}^\ell = - a_{j,i}^\ell, \qquad \mbox{for all $i,j,\ell$}.
\]
Note that a point $[v]$ of $\p(V)$ corresponding to a vector
$v=(v_1,\ldots,v_{2k})\in V$ lies in $X_\phi = \rD_{m-1}(\phi) = \rD_{m-1}(F_\phi)$
if and only if there is a vector $u = (u_1,\ldots,u_m) \in U$ such
that $v$ lies in
$\ker(\tra F_{[u]})$, i.e. such
that:
\[
\begin{pmatrix}
\sum_{j=1}^{2k} a_{1,j}^1 v_j & \cdots & \sum_{j=1}^{2k} a_{1,j}^m v_j\\
\vdots && \vdots \\
\sum_{j=1}^{2k} a_{2k,j}^1 v_j & \cdots & \sum_{j=1}^{2k} a_{2k,j}^m v_j \\
\end{pmatrix} \cdot
\begin{pmatrix}
u_1 \\
\vdots \\
u_m
\end{pmatrix} = 0.
\]
This happens if and only if:
\begin{equation}
  \label{coordinate}
  \sum_{j,\ell} a_{i,j}^\ell v_j u_\ell = 0, \qquad \mbox{for all $i$}.  
\end{equation}

\begin{remark}
  Let $\phi = \phi_m$ be a map $\OO_{\p^{2k-1}}^m \to
  \Omega_{\p^{2k-1}}(2)$.
  Then the degree of $X_m = X_{\phi_m}$ is:
  \begin{align} \label{grado}
    & \deg(X_m) = \sum_{i=0}^{2k-m} (-1)^{i}{2k -1-i\choose {m-1}}, 
  \end{align}
    see \cite{bazan-mezzetti}.
    Indeed, since $X_m$ has the expected codimension, it suffices to compute the $(2k-m)$-th Chern class of $\Omega_{\p^{2k-1}}(2)$.
\end{remark}

\subsection{Skew-symmetric forms and pfaffian hypersurfaces}

%%% We consider again two $\kk$-vector spaces $V$ and $U$ respectively of
%%% dimension $2k$ and $m$, and a map $\phi : U \to \wedge^2 V^*$, or in
%%% other words a map $\phi$ as in \eqref{phi}.
Let $\phi$ be as in \eqref{phi}.
Note that we can consider $\phi$ also as a map $U \to V^* \ts V^*$,
i.e. as an element of $\Hom_\kk(U,V^* \ts V^*)$.
In turn, this vector space is naturally isomorphic to 
$\Hom_\kk(U \ts V,V^*)$.
So, in view of the canonical isomorphism:
\[
\Hom_\kk(U \ts V,V^*) \cong
\Hom_{\p(U)}(V \ts \OO_{\p(U)}(-1),V^* \ts \OO_{\p(U)}),
\]
the map $\phi$ gives a matrix of linear forms:
\[
  M_{\phi} : V \ts \OO_{\p(U)}(-1) \to  V^{*} \ts \OO_{\p(U)}.
\]
The fact that the map $\phi$ lies in
$\Hom_\kk(U,\wedge^2 V^*)$ implies that the matrix  $M_\phi$ lies
in $\HH^0(\p(U), \wedge^2 V^* \ts \OO_{\p(U)}(1))$, that is, the
matrix $M_\phi$ is skew-symmetric.
This means that the transpose $\tra M_\phi$, once twisted by $\OO_{\p(U)}(-1)$, equals $-M_\phi$
(we still write $\tra M_\phi = -M_\phi$).
Therefore, the determinant of the matrix $M_\phi$ is the square of a
homogeneous polynomial of degree $k$, which is called the pfaffian
$\Pf(M_\phi)$ of $M_\phi$.
If the form $\Pf(M_\phi)$ is non-zero, we set $Y_\phi$ for the hypersurface in $\p(U)$,
defined by $\Pf(M_\phi)$. So the degree of $Y_\phi$ is $k$.
Set $Y=Y_\phi$, and define the hyperplane divisor class $H_Y$ as
$c_1(\OO_Y(1))$, where $\OO_Y(1)$ is the restriction to $Y$ of $\OO_{\p(U)}(1)$.
A different way to see the hypersurface $\Pf(M_\phi)$ is as the
intersection of $\p(U)$ and the pfaffian $\Pf$ in $\p(\wedge^2 V^*)$,
where $\p(U)$ is embedded in $\p(\wedge^2 V^*)$ by $\phi$ and
$\Pf$ is the hypersurface $\check{\G}(2,V) \subset \p(\wedge^2 V^*)$
dual to $\G(2,V) \subset \p(\wedge^2 V)$.

Let us now assume that $Y=Y_\phi$ is integral.
Note that the singular locus of $Y$ 
contains the further degeneracy locus $\rD_{2k-4}(M_\phi)$ of the matrix
$M_{\phi}$. Thus the assumption that $Y$ is integral implies that
$M_\phi$ is generically of corank $2$ over $Y$.
We define thus the rank $2$ sheaf $\cE_\phi = \cok(M_\phi)$ supported over $Y$.
In view of \cite[Theorem B, Corollary 2.4]{beauville:determinantal}, one can
consider the sheaf $\cE_\phi$ as a rank-$2$ torsion-free sheaf supported on $Y$, such that: 
\begin{align}
  \nonumber & c_1(\cE_\phi) = (k-1)H_Y, \\
  \label{ACM} & \HH^i(Y,\cE_\phi(t)) = 0, && \mbox{for all $t\in \Z$ and  $0< i< m-2$}.
\end{align}
The sheaf $\cE_\phi$ is ACM (which stands for arithmetically
Cohen-Macaulay), hence reflexive, see \cite[Proposition 2.3]{hartshorne-casanellas:biliaison}.
We have $V^{*} \cong \HH^0(Y,\cE_\phi)$, hence a natural map
$\ppp(\cE_\phi) \to \p(V)$, which we denote by $q$.

The scroll $X_{\phi}$ is the image in
$\p(V)$ of the map $q$.
Therefore, a Palatini scroll is generically ruled over the pfaffian
hypersurface $Y_\phi$ of $\p(U)$. 
The hypersurface $Y=Y_\phi$ is called the {\it base} of the scroll $X=X_\phi$.
We have thus a diagram:
\[
Y \xleftarrow{p} \ppp(\cE_{\phi}) \xrightarrow{q} X \subset \p(V),
\]
where $p$ is the scroll map.

Let us rephrase the situation in coordinates.
Setting $M=M_\phi$, we note that a point $[u]$ of $Y$ corresponds to a 
vector $u = (u_1,\ldots,u_m) \in U$ such that $M_{[u]}$ has not maximal rank.
This means that there is a vector $v = (v_1,\ldots,v_{2k})$ in $V$
such that \eqref{coordinate} holds. In other words, 
an element $[v,u] \in \ppp(\cE_\phi)$, is the class of a pair $(v,u)$
where $v$ is a vector lying in the space $\cok(M_{[u]})$.
The maps $p$ and $q$ are respectively the two projections of $[v,u]$ to $u$ and $v$.

Given a Palatini scroll $X$, we denote by $H$ the divisor
class $c_1(\OO_X(1))$, where $\OO_X(1)$ is obtained restricting
$\OO_{\p(V)}(1)$ to $X$.
Note that $H$ is the first Chern class of the tautological line bundle on the variety
$\ppp(\cE_{\phi})$, namely we have:
\[
p_*(q^*(\OO_X(H))) \cong \cE_{\phi}.
\]

\begin{lemma} \label{iso}
  If the locus $\rD_{m-2}(\phi)$ is empty, then $q:\ppp(\cE_\phi)
  \to X_\phi$  is an isomorphism.
\end{lemma}

\begin{proof}
  Set $M=M_\phi$, $\cE=\cE_\phi$.
  Recall that a point of $\ppp(\cE)$ can be written as the class $[v,u]$ of a
  pair $(v,u)$, where the class $[u]$ of the vector $u$ of $U$ lies in
  $Y$ (so that $M_{[u]}$ is degenerate) and the vector $v$ of $V$ lies in $\cok(M_{[u]})$.
  The map $q$ associates to $[v,u]$ the class $[v]$.
  An inverse to the map $q$ can be thus constructed as follows.
  Given a point $[v]$ of $X$, we must have
  $0 \neq \ker(\phi_{[v]}) \subset U$,
  for $X$ consists of the points of $\p(V)$ where
  $\phi$ is degenerate. Now, since the degeneracy locus
  $\rD_{m-2}(\phi)$ is empty, this kernel must be generated by a
  single vector $u$. Note that $v$ naturally lies in $\cok(M_{[u]})$
  since $M_{[u]}(v)=\phi_{[v]}(u)=0$.
  Therefore we can associate the pair $[v,u]$ to
  $[v]$, and this gives an inverse to $q$.
\end{proof}

\begin{remark}
  Note that the hypersurface $Y = Y_{\phi} \subset \p(U)$ which is the
  base of the scroll $X_{\phi}$ is
  singular as soon as $m\geq 7$. Indeed, the singular locus of $Y$
  contains the degeneracy locus $\rD_{2k-4}(M_\phi)$ of the matrix $M_{\phi}$,
  i.e. the subscheme of $\p(U)$ where the rank of $M_{\phi}$ drops at
  least by four. This locus has codimension at most $6$, so it is not
  empty as soon as $m \geq 7$.

  By the same reason,
  the sheaf $\cE = \cE_{\phi}$ over the hypersurface $Y_{\phi}$ is not
  a vector bundle as soon as $m\geq 7$. Indeed, the locus where $\cE$ is
  not locally free contains $\rD_{2k-4}(M_\phi)$. Recall that anyway
  $\cE$ is ACM (hence reflexive).
  In fact, if $X$ is smooth then $\rD_{m-2}(\phi)=\emptyset$, so $X
  \cong \ppp(\cE)$ by Lemma \ref{iso}.
  So $\cE$ is a B\v anic\v a sheaf in  the sense of \cite{ballico-wisniewski}.
\end{remark}

\begin{remark} One can ask which hypersurfaces of degree $k > 1$ of $\p(U)$ arise as the
  pfaffian of a matrix of the form $\phi$
  (we say that $Y$ is a linear pfaffian in this case).
  Let us collect some well-known answers to this question. 
  \begin{description}
  \item[$m=3$] Any smooth curve $Y \subset \p(U)$ is a linear pfaffian 
  by \cite{beauville:determinantal}, see
  also \cite{buckely-kosir:pfaffian}.
\item[$m=4$]   A general surface $Y \subset \p(U)$ is  a linear
  pfaffian if and only if $k$ is at most $15$.
  Otherwise, surfaces of degree $k$ which are a linear pfaffian fill in only a
  closed subset of $\p(\Sym^{k} U^{*})$.
  We refer to
  \cite{beauville:determinantal}.
  In particular when $m=4,k=3$, one knows that every cubic surface is a linear pfaffian.
  We refer again to \cite{beauville:determinantal} and to
  \cite{fania-mezzetti} for the singular case.
  The cases with $m=4,k \leq 5$ are also treated in
  \cite{dani:cubic:ja}, \cite{dani-chiantini:to-appear},
  \cite{dani-chiantini:linear-pfaffian:to-appear}.
  \item[$m=5$] A general threefold $Y \subset \p(U)$ is a linear pfaffian if and
  only if $k\leq 5$.
  We refer to \cite{beauville:cubic}, \cite{markushevich-tikhomirov},
  \cite{kuznetsov:v14} for the case $k=3$,
  \cite{madonna:complutense} and 
  \cite{iliev-markushevich:degree-quartic} for the case $k=4$,
  \cite{chiantini-madonna} for $k=5$,
  \cite{kumar-rao-ravindra:three-dimensional}, \cite{chiantini-madonna:international} for $k\geq 6$.
  \item[$m=6$] A general fourfold of degree $k$ is not pfaffian unless
  $k=2$, see \cite{kumar-rao-ravindra:hypersurfaces}.
  \item[$m\ge7$] No smooth hypersurface in $\p(U)$ is a linear
    pfaffian, see \cite{kleppe:pfaffian}.
  \end{description}
\end{remark}

\subsection{Kernel, cokernel and normal bundles}

Consider a Palatini scroll $X=X_\phi$ defined by a map $\phi$ as in
\eqref{phi}, and assume that the degeneracy locus $\rD_{m-2}(\phi)$ is
empty (which is the case if $X$ is smooth).
Then it is well-known (see for instance \cite{pragacz:enumerative}) that 
restricting $\phi$ to $X$ we obtain two locally free sheaves on $X$,
$\ker(\phi_{|X})$ and $\cok(\phi_{|X})$, respectively of rank $1$ and $2k-m$.
We restrict $\tra \phi$ to $X$ and we set  
set $\sL_\phi = \cok(\tra \phi_{|X})$. Then $\cok(\tra \phi)$ is
the extension by zero of $\sL_\phi$ to $\p(V)$.
We get $\ker(\tra \phi_{|X})^* \cong \cok(\phi_{|X})$ and 
$\sL_\phi$ is a line bundle on $X$ isomorphic to $\ker(\phi_{|X})^*$.

Moreover, the normal bundle $\cN$ of $X$ in $\p(V)$ is isomorphic to
$\ker(\tra \phi_{|X})^* \ts \sL_\phi$. We have thus the exact sequence
defined on $X$:
\begin{equation}
  \label{es:normal}
  0 \to \OO_X \to U \ts \sL_\phi \xr{\alpha} \Omega_{\p(V)}(2) \ts \sL_\phi \to \cN \to 0,
\end{equation}
where the map $\alpha$ is simply $\phi_{|X}$ tensored with $\sL_\phi$.
This gives rise to the short exact sequences:
\begin{gather}
  \label{short1}
   0 \to \OO_X \to U \ts \sL_\phi \to \im(\alpha) \to 0, \\
  \label{short2}
   0 \to \im(\alpha) \to \Omega_{\p(V)}(2) \ts \sL_\phi \to \cN \to 0. 
\end{gather}

Note that, under the natural isomorphism
$\Hom_\kk(U,V^*\ts V^*) \cong \Hom_\kk(V,U^*\ts V^*)$, the map $\phi$
gives rise to a linear map
\begin{equation}
  \label{fphi}
  f_\phi: V \to U^* \ts V^*. 
\end{equation}
Note also that the map $f_\phi$ agrees with the map obtained by taking global
sections of  \eqref{phitrasposta} twisted by $ \OO_{\p(V)}(1)$.

\subsection{Assumptions on skew-symmetric forms} \label{hyp}

We summarize here the hypothesis that we will need to formulate our
results, and some of their basic consequences.
Let again $U$ and $V$ be $\kk$-vector spaces respectively of
dimension $m$ and $2k$, with $k\geq 3$, and let $\phi$ be a map $U \to \wedge^2 V^*$, or in
other words let $\phi$ be as in \eqref{phi}.

The assumptions that we will need on the map $\phi$ are the following.
\begin{enumerate}
\item \label{hyp:smooth} The subscheme $X \subset \p(V)$ defined as $\rD_{m-1}(\phi)$
  is a smooth irreducible variety of codimension $2k-m$ in $\p(V)$, i.e. $\dim(X)=m-1$.
  This is an open condition, and it implies $\rD_{m-2}(\phi) = \emptyset$. This in turn forces:
  \[
  m \leq k+1.
  \]
  In the above range, the condition takes place for general $\phi$.
\suspend{enumerate}
  As a consequence of this we get the following two conditions:
\resume{enumerate}
\item \label{not-all} the matrix $M_\phi$ is generically of maximal
  rank, so that $\Pf(M_\phi)$ does define a hypersurface $Y_\phi$  of $\p(U)$.
  Equivalently, the image of $\p(U)$
  in $\p(\wedge^2 V^*)$ is not contained in
  the pfaffian hypersurface;
\item \label{hyp:injective} the map $f_\phi$ of \eqref{fphi} is
  injective.
\suspend{enumerate}
We further assume:
\resume{enumerate}
\item \label{hyp:sing} the hypersurface $Y=Y_\phi$ which is the base
  of $X$ is smooth in codimension $4$.
  This is an open condition that 
  takes place for general $\phi$.
  In view of a result of Grothendieck, \cite[Expos\'e XI, Corollaire
  3.14]{sga2}, this condition implies that $Y$ is locally
  factorial (see \cite[Corollary 14]{hansen:deligne}).
  Clearly the hypersurface $Y$ is also integral in this case.
\end{enumerate}

\begin{remark}
In the assumption that $X$ has expected codimension, we have
$X \neq \emptyset$ as soon as $m\geq 2$.
The scheme $X$ is defined by the minors of order $m$ of the matrix
$F_\phi$, so $X$ is cut scheme-theoretically by homogeneous forms of
degree $m$.
In particular we have:
\begin{equation}
\label{degree} \HH^0(\p(V),\cI_X(m-1))=0, 
\end{equation}
if $m\geq 2$.
\end{remark}

\section{Image and cokernel sheaves}

In this section we work with the image and cokernel sheaves associated
to a map $\phi$ as in \eqref{phi}.
Our goal is to show that 
we can read off $\OO_{Y_\phi}(1)$ as the cokernel of $\tra \phi$,
where $Y_\phi \subset \p(U)$ is the base of the scroll $X_\phi$.
This will hold under the assumption that $m\geq 4$ or that $k\geq 4$ if $m=3$.
We first need a vanishing result (proved by induction) on the image of $\tra \phi$.
We denote by $I_\phi$ the torsion-free coherent sheaf defined as the image of $\tra \phi$.

\begin{lemma} \label{vanishing}
  Let $m\geq 2$ and $k\geq 2$ be integers, $V$ and $U$ be vector spaces of
  dimension respectively $2k$, and $m$.
  Let $X_\phi$ be the degeneracy locus $\rD_{m-1}(\phi)$ of a map $\phi: U \ts
  \OO_{\p(V)} \to \Omega_{\p(V)}(2)$
  satisfying the conditions listed in Section \ref{hyp}.
  Then we have:
  \begin{align}
    \label{sezioni}
    & \HH^0(\p(V),I_\phi)=0.
  \end{align}
\end{lemma}

\begin{proof}
  For the sake of this lemma, we let $\phi_m$ be a map $\phi: U_m \ts
  \OO_{\p(V)} \to \Omega_{\p(V)}(2)$, where $U_m$ has dimension $m$.
  Correspondingly, we denote $X_m = \rD_{m-1}(\phi_m)$,
  and $\sL_m = \cok(\tra \phi_m)$, $I_m = \im(\tra
  \phi_m)$, $F_m = F_{\phi_m}$.
  Restricting $\phi_m$ to an $(m-1)$-dimensional subspace $U_{m-1}$ of $U_m$
  we obtain a map $\phi_{m-1}$. Transposing $\phi_{m-1}$ we obtain
  thus the degeneracy locus $X_{m-1}$ and the sheaves $\sL_{m-1}$
  and $I_{m-1}$. We also define the pfaffian hypersurface
  $Y_m = Y_{\phi_m} \subset \p^{m-1}$ and the rank $2$ sheaf $\cE_m = \cE_{\phi_m}$.

  We first claim that, for a general choice of the subspace $U_{m-1}$, we may
  suppose that $X_{m-1}$ is still of the expected codimension.
  Indeed, since $X_{m} \cong \ppp(\cE_m)$, and $Y_m$ contains $Y_{m-1}$
  as a hyperplane section, the variety $\ppp(\cE_{m-1}) \cong X_{m-1}$ has
  dimension $m-2$, hence expected codimension in $\p(V)$.

  The surjective map $U^*\ts \OO_{\p(V)} \to \sL_m$ provides nonzero
  maps $\OO_{\p(V)} \to \sL_m$, and, since $\sL_m$ is supported on $X_m$,
  such map factors through a (nonzero) map $\OO_{X_m} \to \sL_m$.
  This map is thus injective, and we have
  the following commutative exact diagram:
  \begin{equation}
    \label{diagramma}
    \xymatrix@-1ex{
      & 0 \ar[d] & 0 \ar[d] & 0 \ar[d] \\
      0 \ar[r] & \cI_{X_m} \ar[r] \ar[d] & \OO_{\p(V)}
      \ar[r] \ar[d] & \OO_{X_m} \ar[r] \ar[d] & 0 \\
      0 \ar[r] & I_m \ar[d] \ar^-{\tra \phi_{m}}[r] &
      U_m^* \ts \OO_{\p(V)} \ar[r] \ar[d] & \sL_m \ar[r] \ar[d] & 0 \\
      0 \ar[r] & I_{m-1} \ar^-{\tra \phi_{m-1}}[r] \ar[d] & U_{m-1}^* \ts
      \OO_{\p(V)} \ar[r] \ar[d] & \sL_{m-1} \ar[r] \ar[d] & 0 \\
      & 0 & 0 & 0 
    }  
  \end{equation}

  Recall that the ideal sheaf of $X_m = \rD_{m-1}(\phi_m)$,
  is generated by the minors of order $m$ of
  the matrix $F_{m}$.
  Note that the same holds for all $2 \leq m'\leq m$.
  Recall also that $X_m$ is not empty if $m\geq 2$, hence we may use
  the vanishing \eqref{degree}.
  Therefore, taking global sections in the leftmost column of \eqref{diagramma}
  and using induction, we obtain the vanishing \eqref{sezioni} as soon
  as we prove it for $m=2$.

  We prove now the vanishing \eqref{sezioni} for $m=2$.
  Note that \eqref{grado} implies $\deg(X_1)=0$, so that $X_1 =
  \emptyset$ i.e. $\phi_1$ vanishes nowhere.
  This implies that $I_1$ is isomorphic to $\OO_{\p(V)}$.
  The leftmost column of \eqref{diagramma} reads in this case:
  \[
  0 \to \cI_{X_2} \to I_2 \to \OO_{\p(V)} \to 0.
  \]

  Since clearly $\HH^0(\p(V),\cI_{X_2})=0$, we obtain 
  $\hh^0(\p(V),I_2)\leq 1$. Note that the value $1$ is attained if and only
  if $I_2$ splits as $\OO_{\p(V)} \oplus \cI_2$.
  But, since $I_2$ is the image of the map $F_\phi$, this is
  contradicted by $\HH^0(\p(V),\cI_{X_2}(1))=0$.
  In turn, the last vanishing takes place in view of \eqref{degree}.
\end{proof}

\begin{lemma} \label{O1}
  Fix the assumptions as above and set $\sL_\phi = \cok(\tra \phi)$.
  Assume moreover $k\geq 4$ in case $m=3$.
  Then we have an isomorphism:
  \[
  \sL_{\phi} \cong q_*(p^*(\OO_{Y_\phi}(1))). 
  \]
\end{lemma}

\begin{proof}
Set $\sL = \sL_\phi$, and recall that $\sL$ is an invertible sheaf on $X$.
Restricting to $X$ the dual  Euler exact sequence on $\p(V)$ and twisting by
$\sL  \ts \OO_{X} (2)$, we obtain the equality
$c_1( \Omega_{\p(V)|X}(2) \ts \sL)=(2k-2)H+(2k-1)c_1(\sL)$.
Further, the sequence \eqref{short1}  gives
$c_1( \im(\alpha) )=m c_1(\sL)$ while
the sequence   \eqref{short2} gives
$c_1(\cN)=c_1( \Omega_{\p(V)|X}(2) \ts \sL)-c_1(\im(\alpha))$.
Combining them we get:
\[c_1(\cN)=(2k-2)H+(2k-m-1)c_1(\sL).\]

On the other hand, from adjunction we know that
$c_1(\cN)=2k H-c_1(X)$.
Because $X$ is generically a scroll over $Y$, we know that
$c_1(X)=2H+p^*(c_1(Y)-c_1(\cE))=2H+p^*((m+1-2k)H_{Y})$ and thus 
 $c_1(\cN)=(2k-2)H+p^*((2k-m-1)H_{Y})$. Hence:
 \begin{equation}
  \label{torsX}
 (2k-m-1)c_1(\sL)=p^*((2k-m-1)H_{Y}).
\end{equation}
This implies that $c_1(\sL)=p^*(c_1(\sM))$ for some line bundle $\sM
\in \Pic(Y)$, and, by  \eqref{torsX},  we get that:
\[
  (2k-m-1)(c_1(\sM)-H_{Y})=0 \qquad \mbox{in $\Pic(Y)$.}
\]

Now recall that, for $m\ge 5$, the group $\Pic(Y)$ is generated by
$H_Y$ by the theorem Grothendieck-Lefschetz, while, for $m=4$, the
cokernel of the natural restriction map $\Pic(\p(U)) \to \Pic(Y)$ is torsion-free (we refer for instance to \cite{badescu:picard}).
It follows that $c_1(\sM)=H_Y$, and we deduce $\sL \cong p^*(\OO_{Y}(1))$.

It remains to prove the statement for $m=3$.
Note that we have $\hh^0(Y, \sM)=\hh^0(X, p^*(\sM))=\hh^0(X, \sL)$,
and Lemma \ref{vanishing} easily implies that $\hh^0(X, \sL) \geq 3$.
Note also that $\deg(\sM)=\deg(\sL)=k$, by \eqref{torsX}.

Hence the linear system $|\sM|$ is a  $g^s_{k}$ on the curve $Y$, with $s\ge 2$.
Now we use a well-known result of Castelnuovo
(see \cite{castelnuovo:curve}, we refer to \cite[Theorem 2.11]{ciliberto:bologna} 
and \cite{accola:castelnuovo} for a modern treatment).
Namely, any linear series $g^s_k$ on a smooth plane curve $Y$ of
degree $k \geq 4$ (here we need our hypothesis) and with $s \geq 2$ must coincide with the
hyperplane linear system $|\OO_Y(1)|$ (in particular $s=2$).
This gives $c_1(\sM)=H_Y$, and we have proved $\sL=p^*(\OO_{Y}(1))$.
Note that we have also proved that, if $m=3$ and $k\geq 4$, we must have $\hh^0(\p(V),\sL)=3$.
\end{proof}

The previous proof suggests the following.

\begin{problem}
  Let $Y = Y_\phi$ be a general pfaffian surface, i.e. $Y$ is given as
  the pfaffian of a general $2k \times 2k$ skew-symmetric matrix of
  linear forms over $\p^3$. Then, is the Picard group of $Y$
  generated by the hyperplane divisor $H_{Y}$, for $k \geq 16$?
\end{problem}

\begin{lemma}
  Fix the hypothesis as in Lemma \ref{vanishing}. Then we have:  
  \[
   \HH^1(\p(V),I_\phi)=0, \qquad \mbox{for all $m\geq 3$.}
  \]
\end{lemma}

\begin{proof}
  We borrow the notation from the proof of Lemma \ref{vanishing}.
  Let us recall that, by \cite[Proposition 1]{bazan-mezzetti}, we have:
  \begin{align}
    \nonumber    & \HH^1(\p(V),\cI_{X_m}(t)) = 0,   && \mbox{unless $m$ is even and
      $t$ equals $m-2$,} \\
    \nonumber & \HH^2(\p(V),\cI_{X_m}) = 0,   && \mbox{if $m \neq 3$.} 
  \end{align}
  Therefore, taking cohomology in the leftmost column of \eqref{diagramma}
  and using induction on $m$,
  we easily get $\HH^1(\p(V),I_m)=0$ for all $m\geq 4$ once we prove
  $\HH^1(\p(V),I_3)=0$.

  To check the last vanishing, recall that, at the end of the proof of
  the previous lemma we have shown $\hh^0(\p(V),\sL_3)=3$ if $k\geq
  4$. However, the equality $\hh^0(\p(V),\sL_3)=3$ holds even
  if $k=3$ for $\sL$ in this case is a line bundle of degree $3$
  supported on $Y$ which is a smooth elliptic curve.
  In any case, by \eqref{sezioni}, we immediately deduce $\HH^1(\p(V),I_3)=0$, and
  the lemma is proved.
\end{proof}

\section{Normal bundle}

%%% Let again $V$ be a $\kk$-vector space of dimension $2k$ with $k\geq 3$,
%%% $U$ be a $\kk$-vector space of dimension $m\geq 3$,
%%% and $\phi$ be a map $U \ts \OO_{\p(V)} \to \Omega_{\p(V)}(2)$.
Let again $\phi$ be as in \eqref{phi} and set $X=X_\phi$.
Assume that $X$ is smooth.
We have noticed that the normal bundle $\cN$ of $X$ in $\p(V)$ fits in
the exact sequence \eqref{es:normal}.

The goal of this section is to compute the dimension of the Hilbert
scheme of Palatini scroll $\cH_m(V)$ at the point represented by $X$, by calculating
the dimension of $\HH^0(X,\cN)$.
The idea is to compute this dimension by pushing down the exact
sequence \eqref{es:normal} to the base variety $Y$.
The key argument relies on a technical result of 
Mohan Kumar-Rao-Ravindra, see \cite{kumar-rao-ravindra:hypersurfaces}.

\begin{theorem} \label{thm:normal}
  Assume that the map $\phi$ satisfies the conditions of Section \ref{hyp}, and let $\cN$ be
  the normal bundle of $X$ in $\p(V)$.
  Then we have, for all $m \geq 4$ and $k\geq m-1$:
  \begin{align}
    \label{normal_0} &   \hh^{0}(X,\cN)= m \left(k (2 k - 1)  - m \right).
\intertext{Furthermore, if $m \geq 5$ or if $m=4,k \leq 4$, we have:}
    \label{normal_1_0} &   \HH^{1}(X,\cN)= 0,
\intertext{while, if $m = 4$ and $k \geq 5$, we have:}
    \label{normal_1_d} &   \hh^{1}(X,\cN)= \frac{2(k-2)(k-3)(k-4)}{3}.
\intertext{Finally, if $m=3$, $k\geq 4$, we have:}
    \nonumber &   \hh^{0}(X,\cN)= \frac{3k(5k-7)}{2}, \\
    \nonumber & \HH^{1}(X,\cN)=0.
  \end{align}
\end{theorem}

In order to prove the theorem, we will need a little more material.
Let us introduce it in the next subsection.

\subsection{The exterior square of a skew-symmetric matrix of linear forms}

Consider again the matrix $M=M_\phi$ of linear forms on $\p(U)$
associated to $\phi : U \ts \OO_{\p(V)} \to \Omega_{\p(V)}(2)$, and
the pfaffian hypersurface $Y = Y_\phi$ defined by $M$ in $\p(U)$.
The restriction of the map $M$ to
$Y$ provides an exact sequence:
\begin{equation}
  \label{GVE}
  0 \to G \to V^{*} \ts \OO_Y \to \cE \to 0,
\end{equation}
where $G$ is a sheaf of rank $2(k-1)$ on $Y$ defined by the above sequence.

We consider the symmetric square of \eqref{GVE}, and we twist
the result by $\OO_Y(1)$, whereby obtaining a
four-term exact sequence of the form:
\begin{equation}
  \label{symm}
  0 \to \wedge ^2 G(1) \xr{\gamma} \wedge^2 V^{*} \ts \OO_Y(1) \xr{\delta} V^{*} \ts \cE(1) \xr{\eta} \Sym^2 \cE(1) \to 0.
\end{equation}

We will make use of the setup of
\cite{kumar-rao-ravindra:hypersurfaces}.
Taking the exterior square of the matrix $M$ and twisting by
$\OO_{\p(U)}(1)$, we obtain the exact sequence:
\begin{equation}
  \label{Fcal}
  0 \to \wedge^2 V \ts \OO_{\p(U)}(-1) \xr{\wedge^2 M} \wedge^2 V^{*} \ts \OO_{\p(U)}(1) \to \cF(1) \to 0,
\end{equation}
for some  cokernel sheaf $\cF$ supported on $Y$.

\begin{lemma} \label{lemmone}
  Fix the setup as above. Then, for all $m\geq 4$ we have:
\begin{align}
  \nonumber  & \hh^0(Y,\im(\delta))= m k (2 k - 1) -1.
  \intertext{Moreover, if $m \neq 4$ we have the vanishing:}
  \nonumber & \HH^1(Y,\im(\delta))= 0,
\intertext{while for $m=4$, we have the vanishing $\HH^2(Y,\im(\delta))=
  0$ and the equality:}
  \nonumber & \hh^1(Y,\im(\delta)) = \hh^0(Y,\omega_Y),
\intertext{which equals zero if $k\leq 3$ and equals $\frac{(k-1)(k-2)(k-3)}{6}$
if $k\geq 4$. Finally, if $m=3$ we have:}
  \nonumber & \hh^0(Y,\im(\delta))= m k (2 k - 1) +
  \frac{(k-1)(k-2)}{2} -1, \\ %% FIX controllare
  \nonumber & \HH^1(Y,\im(\delta))= 0.
\end{align}
\end{lemma}

\begin{proof}
Looking at the diagram in the proof of \cite[Lemma 2.1]{kumar-rao-ravindra:hypersurfaces}, we see that the 
image of the map $\delta$, (which would be denoted in the notation
of that paper by $\overline{\mathcal{F}}(1)$) fits into the exact sequence:
\begin{equation}
  \label{Fbar}
  0 \to \cL(1) \to \cF(1) \to \im(\delta) \to 0,
\end{equation}
where $\cL$ is the line bundle on $Y$ provided by [{\it op. cit.}, Lemma 2.1.]
%%%%%%%%%%%%%%%%%%
%% REVISIONE CON AGGIUNTA CLAIM  %%%%%%
%%%%%%%%%%%%%%%%
Note that, even though the hypersurface $Y$ is not smooth in general,
we can still use this lemma since $Y$ is integral and locally
factorial in our assumptions (see the hypothesis \eqref{hyp:sing}),
and the sheaf $\cL$ on $Y$ is also a line bundle in this case. In fact
this will follow by \cite[Propositions 1.1 and
1.9]{hartshorne:stable-reflexive}, once we prove the following
claim. 

\begin{my-claim} \label{claim0}
 The above sheaf  $\cL$ on $Y$ is reflexive of rank $1$.
\end{my-claim}

\begin{proof}[Proof of the claim]
We dualize \eqref{Fcal}.
Since the matrix $M$ is skew-symmetric, the matrix $\wedge^2 M$ is
symmetric, so we have an isomorphism:
\[
\tau : \cok(\wedge^2 M) \cong \cok(\tra \wedge^2 M).
\]

Grothendieck duality implies:
\[
\cok(\tra \wedge^2 M) \cong \EExt^1_{\p(U)}(\cF, \OO_{\p(U)}) \cong \cF^*(k-1).
\]
We also know that $\cok(\wedge^2 M)=\cF(1)$ so that:
\begin{equation}
  \label{c1F}
  \cF(1)  \cong \cF^*(k-1). 
\end{equation}
By restricting \eqref{Fcal} to $Y$ we  get an exact sequence
\begin{equation}
  \label{FcalY}
  0 \to \cF^*(-1)  \to \wedge^2 V \ts \OO_{Y}(-1) \xr{\overline{\wedge^2 M}} \wedge^2 V^{*} \ts \OO_{Y}(1) \to \cF(1) \to 0
\end{equation}
Define the torsion-free sheaf $P$ as the image of $\overline{\wedge^2
  M}$.
Twisting \eqref{FcalY} by $ \OO_Y(k)$ and using \eqref{c1F} we see that $\cF(1)$ 
 is contained  in $\wedge^2 V \ts \OO_Y(k-1)$ with
cokernel $P(k)$. Combining this with  the sequence \eqref{Fbar}  we
get  an exact sequence:
\begin{equation}
0 \to \cL(1) \to \wedge^2 V \ts \OO_Y(k-1)  \to P' \to 0,
\end{equation}
where $P'$  sits in the following exact sequence:
\[
0 \to \im(\delta) \to P' \to P(k) \to 0.
\]
Hence $ P'$ is torsion-free. Moreover $\wedge^2 V \ts \OO_Y(k-1)$ is
locally free and thus by  \cite[Proposition
1.1]{hartshorne:stable-reflexive} it follows that $\cL(1)$ is
reflexive and of rank $1$. 
\end{proof}

\begin{my-claim} \label{claim1}
  Fix the setup as above. Then we have an isomorphism:
  \[\cL \cong \OO_Y(-1).\]
\end{my-claim}

\begin{proof}[Proof of the claim]
Recall that $c_1(\cE)=(k-1)H_Y$.
To calculate $c_1(\cL)$, first note that:
\begin{equation}
  \label{c1}
  c_1(\cL)=c_1(\cF)-c_1(\im(\delta)(-1)),  
\end{equation}
and splitting \eqref{symm} into short exact sequences we easily compute:
\[
c_1(\im(\delta)(-1)) = (k-1)(2k-3)H_Y.
\]

From  \eqref{c1F} we deduce that:
\[
c_1(\cF(1))=c_1(\cF^* (k-1)),
\]
which implies, since $\rk(\cF)=2(2k-1)$, the equality:
\[
c_1(\cF)=(2k-1)(k-2)H_Y.
\]
Now \eqref{c1} implies $c_1(\cL)=-H_Y$ and we are done.
\end{proof}

We can now conclude the proof of the lemma.
We take cohomology of the exact sequence \eqref{Fcal}. 
For $m\geq 3$, we get:
\begin{align*}
  & \hh^0(Y,\cF(1))= m {2k\choose{2}}  = m k (2 k - 1), \\
  & \hh^i(Y,\cF(1))=0,  \quad \mbox{for $i\geq 1$.}
\end{align*}
In order to get the desired formulas, we take now cohomology of \eqref{Fbar}
and we use Claim \eqref{claim1}.
Then the result follows, once we note that $\HH^{1}(Y,\OO_{Y})=0$ for
$m\geq 4$ and $\hh^{1}(Y,\OO_{Y})=\frac{(k-1)(k-2)}{2}$ for $m=3$.
\end{proof}

\subsection{Projecting the normal bundle on the pfaffian hypersurface}

Let $X=X_\phi$ be a scroll defined by a map $\phi$ as in \eqref{phi},
and satisfying the conditions of Section \ref{hyp}.
The idea to prove Theorem \ref{thm:normal} is to push
the exact sequence defining the normal bundle $\cN$ of the scroll
$X=X_\phi$ in $\p(V)$ down
to the pfaffian hypersurface $Y=Y_\phi$. We do this in the next lemma.

\begin{lemma} \label{alpha}
  Fix the setup as in Theorem \ref{thm:normal}. Then the following equality holds for all
  $m\geq 3$ and $k\geq 4$:
  \begin{align}
    \label{imalpha0} & \hh^0(X,\im(\alpha)) = m^2-1,
    \intertext{and, if $m\geq 5$, and  $(m,k) \neq (5,5)$ we also have:}
    \label{imalpha1} & \HH^{1}(X,\im(\alpha)) = 0, \\
    \label{imalpha2} & \HH^{2}(X,\im(\alpha)) = 0.
     \intertext{If $(m,k) = (5,5)$, then the vanishing \eqref{imalpha1} holds and:}
    & \label{speciale}  \hh^{2}(X,\im(\alpha)) = 1.  
    \intertext{\indent Assume now $m=4$. Then the same vanishing
      results hold if
      $k\leq 3$. If $(m,k) = (4,4)$, then  \eqref{imalpha2} holds and:}
    &  \label{speciale2} \hh^{1}(X,\im(\alpha)) = 1,  
    \intertext{while if $m=4$ and $k\geq 5$, then  \eqref{imalpha1} holds and:}
    &  \hh^{2}(X,\im(\alpha)) = \frac{(k-2)(k-3)(k-5)}{2}.
  \end{align}
\end{lemma}

\begin{proof}
  Note that, since $\sL = \cok(\tra \phi) \cong q_*(p^*(\OO_Y(1)))$ by Lemma \ref{O1}, in the
  sequence \eqref{es:normal} we may identify the inclusion $\OO_{X}
  \to U\ts \sL$  with the map
  $q_*(p^*(\beta))$, where $\beta$ fits in the Euler exact sequence on
  $\p(U)$, restricted to $Y$:
  \begin{equation}
    \label{eulerU}
    0 \to \OO_{Y} \xr{\beta} U \ts \OO_{Y} (1) \to (\cT_{\p(U)})_{|Y} \to 0.    
  \end{equation}

  Therefore the image of $\alpha$ is isomorphic to
  $q_*(p^*((\cT_{\p(U)})_{|Y}))$ and we deduce, for all $i$:
  \begin{equation}
    \label{alpha-T}
    \HH^{i}(X,\im(\alpha)) \cong \HH^{i}(Y,(\cT_{\p(U)})_{|Y}).    
  \end{equation}

  We use now the exact sequence defining $Y$ in $\p(U)$, tensored with
  ${\cT_{\p(U)}}_{|Y}$:
  \[
  0 \to \cT_{\p(U)}(-k) \to \cT_{\p(U)} \to (\cT_{\p(U)})_Y \to 0.
  \]
  Note that $\hh^0(\p(U),\cT_{\p(U)})=m^2-1$, while
  $\HH^i(\p(U),\cT_{\p(U)})=0$ for all $0<i<m-1$.
  By Bott's theorem, we have, for such $i$:
  \[
  \hh^i(\p(U),\cT_{\p(U)}(-k))=
  \left\{
  \begin{array}{ll}
    1 & \mbox{if $i=m-2$ and $k=m$}, \\
    0 & \mbox{otherwise}.
  \end{array} \right.
  \]
  
  This implies \eqref{imalpha0} (indeed $m\geq 3$ and $k\geq 4$).
  We also get \eqref{imalpha1} for $m\geq 5$ and for $m=4,k\neq
  4$, while $(m,k)=(4,4)$ gives \eqref{speciale2}.
  Likewise we get \eqref{imalpha2} for $m\geq 4$, except $(m,k) =
  (5,5)$, in which case we get \eqref{speciale}.
  
  To conclude the proof, we look at the case $m=4$. We note that by Serre duality:
  \begin{equation}
    \label{alpha-Omega}
    \HH^{2}(X,\im(\alpha)) = \HH^0(Y,\Omega_{\p(U)|Y} \ts \omega_Y)^*.
  \end{equation}
  It is easy to see that
  $\hh^0(Y,\Omega_{\p(U)|Y} \ts \omega_Y)$ takes value
  $(k-2)(k-3)(k-5)/2$ if $k\geq 5$, or
  zero if $k\leq 4$.
  This finishes the proof.
\end{proof}

Set the notation as in Lemma
\ref{lemmone}. We have then:
\begin{lemma} \label{projection}
  Let $m\geq 3$ and $k\geq m-1$ be integers, and fix the setup as above.
  Then, for all integers $i\geq 0$, there is a natural isomorphism:
  \begin{equation}
    \label{delta-Omega}
    \HH^{i}(X,\Omega_{\p(V)}(2) \ts \sL) \cong \HH^{i}(Y,\im(\delta)).
  \end{equation}
\end{lemma}

\begin{proof}
  We first write the dual Euler sequence on $\p(V)$, twisted by $\sL
  \ts \OO_{\p(V)}(2)$:
  \[
  0 \to \Omega_{\p(V)}(2) \ts \sL \to V^* \ts \OO_{\p(V)}(1) \ts \sL \xr{\xi} \OO_{\p(V)}(2) \ts \sL \to 0.
  \]
  
  Recall now the natural embedding $q:\ppp(\cE)\to \p(V)$ and the
  canonical projection $p : \ppp(\cE) \to Y$.
  For all $t\geq 0$ and all $\ell$, we have the natural isomorphism:
  \[
  p_{*}q^{*}(\OO_{\p(V)}(t)\ts \sL^{\ell}) \cong \Sym^{t} \cE (\ell).
  \]
  
  Moreover, applying the functor $p_{*}q^{*}$ to the map $\xi$, we find the
  natural projection:
  \[
  \eta :V^{*} \ts \cE(1) \to \Sym^{2} \cE (1),
  \]
  appearing in the exact sequence \eqref{symm}.
  Therefore we have, for all integers $i$:
  \[
  \HH^{i}(X,\Omega_{\p(V)}(2) \ts \sL) \cong \HH^{i}(X,\ker(\xi)) \cong \HH^{i}(Y,\ker(\eta)) \cong \HH^{i}(Y,\im(\delta)).
  \]
\end{proof}

We prove now the main result of this section.
\begin{proof}[Proof of Theorem \ref{thm:normal}]
By the above claim, taking cohomology of \eqref{short2}, we get the long exact sequence:
    \begin{equation}
      \label{lunga}
      \cdots \to \HH^{i}(X,\im(\alpha)) \to \HH^{i}(Y,\im(\delta)) \to \HH^{i}(X,\cN)
      \to \HH^{i+1}(X,\im(\alpha)) \to \cdots
    \end{equation}
  
  Having all this set up, we use Lemma \ref{alpha} and we apply Lemma
  \ref{lemmone} to obtain, for all
  $(m,k) \neq (4,4)$:
  \[
  \hh^{0}(X,\cN)=m k(2k-1)-1-(m^{2}-1)=m(k(2k-1)-m).
  \]
  So \eqref{normal_0} is proved except for $(m,k) = (4,4)$.
  Likewise we get \eqref{normal_1_0} for $m \neq 4,5$ and for
  $m=4$, $k \leq 3$ as well as for $m=5$, $k \neq 5$.
  We also have \eqref{normal_1_d} for $m=4$ and $k\geq 5$, in view of:
  \begin{align*}
  \hh^{1}(X,\cN) & = \hh^{1}(X,\im(\delta))+ \hh^{2}(X,\im(\alpha)) = \\
  & = \frac{(k-1)(k-2)(k-3)}{6} + \frac{(k-2)(k-3)(k-5)}{2} = \\
  & = \frac{2(k-2)(k-3)(k-4)}{3}.
  \end{align*}
  
  It remains to take care of the cases $(m,k) \in \{ (4,4), (5,5) \}$ and to treat the
  case $m=3$.
  Let us accomplish the first task. We will need the following:

  \begin{my-claim}
    Assume $m\geq 4$ and consider the exact sequences \eqref{Fbar} and \eqref{eulerU}.
    Then the inclusion $\im(\alpha) \to 
    \Omega_{\p(V)}(2) \ts \sL$ given by \eqref{es:normal} can be
    identified with $q_{*}(p^{*}(\zeta))$, where the map $\zeta$
    fits in the following exact commutative diagram:
    \begin{equation} \label{4-4}
    \xymatrix{0 \ar[r] & \OO_{Y} \ar@{=}[d] \ar[r] & U \ts \OO_{Y}(1)
      \ar[d] \ar[r] & 
      (\cT_{\p(U)})_{|Y} \ar[r] \ar^-{\zeta}[d] &  0 \\
      0 \ar[r] & \OO_{Y} \ar[r] & \cF(1) \ar[r] & \im(\delta) \ar[r]  & 0 \\
    }
    \end{equation}
  \end{my-claim}
  \begin{proof}[Proof of the claim]
    First recall the natural isomorphisms
    $\Omega_{\p(V)}(2)\ts \sL \cong q_{*}(p^{*}(\im(\delta)))$ (see the
    proof of Lemma \ref{projection}) and 
    $\im(\alpha) \cong q_{*}(p^{*}((\cT_{\p(U)})_{|Y}))$ (see Claim
    \ref{alpha}).
    Then the map $\im(\alpha) \to 
    \Omega_{\p(V)}(2) \ts \sL$ is induced by some nonzero map $\zeta : 
    (\cT_{\p(U)})_{|Y} \to \im(\delta) $.
    Composing the natural surjection $U \ts \OO_{Y}(1) \to
    (\cT_{\p(U)})_{|Y}$ with $\zeta$, we get a nonzero map
    $U \ts \OO_{Y}(1) \to \im(\delta)$.
    Now we can lift this map to a nonzero map $U \ts \OO_{Y}(1)\to
    \cF(1)$, indeed the group $\Ext^{1}(U \ts \OO_{Y}(1),\OO_{Y})$
    vanishes in the range $m\geq 4$.
    We have thus an induced map $\OO_{Y} \to \OO_{Y}$, which must be a nonzero
    multiple of the identity. This gives the diagram \eqref{4-4}.
  \end{proof}

  Let us now prove Theorem \ref{thm:normal} in case
  $(m,k)=(4,4)$.
  We look at the map $\HH^{1}(X,\im(\alpha)) \to
  \HH^{1}(X,\Omega_{\p(V)}(2) \ts \sL)$ induced by the inclusion $\im(\alpha) \to
  \Omega_{\p(V)}(2) \ts \sL$.
  Under the isomorphism \eqref{delta-Omega} and \eqref{alpha-Omega},
  this map is induced by the map $\zeta$ introduced in the previous claim.
  Thus it suffices to show that $\zeta$ induces an isomorphism of
  $\HH^{1}(Y,(\cT_{\p(U)})_{|Y})$ to 
  $\HH^{1}(Y,\im(\delta))$.
  In order to show this,
  we look at the cohomology of the diagram \eqref{4-4} provided by the above
  claim.
  We get a commutative diagram of the form:
    \[
      \xymatrix{
        \HH^{1}(Y,(\cT_{\p(U)})_{|Y}) \ar[d] \ar[r] & \HH^{2}(Y,\OO_{Y}) \ar@{=}[d]  \\
        \HH^{1}(Y,\im(\delta)) \ar[r] & \HH^{2}(Y,\OO_{Y}).
    }
    \]
    Note that both horizontal arrows are isomorphisms, and the rightmost
    vertical arrow is an isomorphism too.
    Therefore the map $\HH^{1}(Y,(\cT_{\p(U)})_{|Y}) \to
    \HH^{1}(Y,\im(\delta))$
    induced by $\zeta$ is an
    isomorphism, and we are done. Similarly, in the case $(m,k)=(5,5)$, we get a diagram of the form  \[
      \xymatrix{
        \HH^{2}(Y,(\cT_{\p(U)})_{|Y}) \ar[d] \ar[r] & \HH^{3}(Y,\OO_{Y}) \ar@{=}[d]  \\
        \HH^{2}(Y,\im(\delta)) \ar[r] & \HH^{3}(Y,\OO_{Y}).
    }
    \] and as before we can conclude that  the map $\HH^{2}(Y,(\cT_{\p(U)})_{|Y}) \to
    \HH^{2}(Y,\im(\delta))$  induced by $\zeta$ is an  isomorphism.

    \medskip To complete the proof of Theorem \ref{thm:normal}, it
    only remains to check the case $m=3$. 
    Let us compute $\hh^{0}(X,\cN)$ in this case.
    Using Lemma \ref{projection},
    in view of the exact sequences \eqref{Fbar} and \eqref{Fcal}, we calculate:
    \begin{align}
      \label{delta-curve}
      \hh^{0}(X,\Omega_{\p(V)}(2)\ts \sL) & = \hh^{0}(Y,\im(\delta)) = \\
      \nonumber & = \hh^{0}(Y,\cF(1))+\hh^{1}(Y,\OO_{Y})-1 \\
      \nonumber & = \frac{k(13k-9)}{2},
    \end{align}
    and $\HH^{i}(X,\Omega_{\p(V)}(2)\ts \sL)=0$ for $i \geq 1$.
    Moreover, the isomorphism \eqref{alpha-T} still holds for $m=3$.
    So, we easily get $\hh^{0}(Y,\im(\alpha))=8$ and:
    \begin{equation}
      \label{alpha-curve}
      \hh^{1}(Y,\im(\alpha))=\hh^{1}(Y,(\cT_{\p(U)})_{|Y})=(k-2)(k-4).
    \end{equation}
    Therefore, using \eqref{delta-curve} and \eqref{alpha-curve}, we can calculate $\HH^{i}(X,\cN)$ by \eqref{lunga},
    obtaining:
    \[
    \hh^{0}(X,\cN)=\frac{k(13k-9)}{2}+(k-2)(k-4)-8=\frac{3k(5k-7)}{2},
    \]
    and $\HH^{i}(X,\cN)=0$ for $i\geq 1$. This completes the proof of
    Theorem \ref{thm:normal}.
\end{proof}

Our theorem asserts that the cohomology of the normal bundle of $X$ in
$\p(V)$ behaves {\it as expected}. However, the following problem
remains open.

\begin{problem}
  Let $\phi$ be a general  morphism $\OO_{\p(V)}^4 \to
  \Omega_{\p(V)}(2)$, where $\dim(V)=2k$, and let $\cE = \cE_\phi$ the associated bundle
  on $Y=Y_\phi$.
  Then do we have, for all $k\geq 3$, $\HH^2(Y,\Sym^2 \cE(1-k))=0$? In other words, is
  the moduli space of stable sheaves containing $\cE$ smooth at $[\cE]$?
\end{problem}

Up to the authors' knowledge, the above question has a well-known
answer (in the affirmative sense) only for $k \leq 15$, due to
a result of Schreyer-Beauville \cite{beauville:determinantal}.

\section{Injectivity of the map $\rho$}

The results of this section are intended to achieve the proof of our
main result.
Once we have computed the dimension of the Hilbert
scheme in the previous section, it only remains to check that $\rho$
is generically injective (hence birational onto its image).

The idea to prove this is that, in view of Lemma \ref{O1}, 
two maps $\phi$ and $\phi'$ defining same scroll $X$ will have isomorphic cokernel
sheaves $\cok(\tra \phi) \cong \cok(\tra \phi')$.
We prove thus some suitable cohomology vanishing in order to lift this
isomorphism to the image sheaves $I_\phi$ and $I_{\phi'}$ and then to
the whole tangent bundle $\cT_{\p(V)}(-2)$.
This will show that $\phi$ and $\phi'$ give the same point in
$\G(m,\wedge^{2} V^{*})$.
This will prove the theorem, except in case $m=3$, where the image of 
$\rho$ is not dense in the Hilbert scheme $\cH_3(V)$
containing Palatini scrolls. We will thus conclude by computing the
codimension of the image of $\rho$ in $\cH_3(V)$.

Let us now consider a map $\phi$ as in the hypothesis of Section \ref{hyp}, and set
$K_\phi = \ker(\tra \phi)$, ${I}_\phi = \im(\tra \phi)$, ${\sL}_\phi =
\cok(\tra \phi)$.
%%%%%%%%%%%%%%%%%%%%%%%%%%%%%%%%%%%%%%%%%
%%%
%% REVISIONE LEMMA 4.1 CON DIMOSTRAZIONE CORRETTA PER m=3 %%%%%%%%%%%
%%%%%%%%%%%%%%%%%%%%%%%%%%%%%%%%%%%%%%%%%
\begin{lemma} \label{lemmone2}
  Fix the hypothesis as in Section \ref{hyp}.
  Then we have:
  \[
     \HH^1(\p(V),K_\phi \ts \Omega_{\p(V)}(2))=0, \qquad \mbox{for all $m\geq 3$.}
  \]
\end{lemma}

\begin{proof}
  Recall the exact sequences:
  \begin{align}
    \label{seq:I} & 0 \to {I}_\phi \to U^* \ts \OO_{\p(V)} \to {\sL}_\phi \to 0, \\
    \label{seq:K} & 0 \to K_\phi \to \cT_{\p(V)}(-2) \to {I}_\phi \to 0.
  \end{align}
  We consider the dual Euler sequence on $\p(V)$, twisted by
  $K_{\phi}(2)$, and take global sections. We obtain an inclusion:
  \[
    \HH^0(\p(V),K_{\phi} \ts \Omega_{\p(V)}(2))  \subset  V^* \ts \HH^0(\p(V),K_{\phi}(1)).
  \]
  Note that, taking cohomology of \eqref{seq:K} and \eqref{seq:I} twisted by
  $\OO_{\p(V)}(1)$ we obtain, respectively, a map $V \to \HH^0(\p(V),I_{\phi}(1))$
  and a map $\HH^0(\p(V),I_{\phi}(1)) \to U^* \ts V^*$. The composition
  of these two maps is the linear map $f_\phi$ in  \eqref{fphi}, which is
  injective in our hypothesis. Therefore we have $\HH^0(\p(V),K_{\phi} (1)) = 0$ and we deduce the vanishing:
  \begin{equation}
    \label{H0K}
    \HH^0(\p(V),K_{\phi} \ts \Omega_{\p(V)}(2))  = 0.
  \end{equation}
  
  Tensoring the exact sequence \eqref{seq:K} by $\Omega_{\p(V)}(2)$, since
  $\HH^0(\p(V),\cT_{\p(V)} \ts \Omega_{\p(V)}) \cong \kk$ and
  $\HH^1(\p(V),\cT_{\p(V)} \ts \Omega_{\p(V)}) =0 $, using \eqref{H0K} we get the exact
  sequence:
  \[
  0 \to \kk \to \HH^0(\p(V),I_{\phi} \ts \Omega_{\p(V)}(2)) \to \HH^1(\p(V),K_{\phi} \ts \Omega_{\p(V)}(2)) \to 0.
  \]
  In order to conclude the proof of the lemma, we have thus to prove:
  \begin{equation}
    \label{H1=1}
    \hh^0(\p(V),I_{\phi} \ts \Omega_{\p(V)}(2))=1.
  \end{equation}
  
  The rest of the proof is devoted to show the above equality.
  First note that, tensoring the dual Euler sequence on $\p(V)$ by the
  ideal sheaf $\cI_X(2)$ and taking cohomology, we obtain:
  \begin{equation}
    \label{IXOMEGA}
    \HH^i(\p(V),\cI_X \ts \Omega_{\p(V)}(2))=0,  \qquad  \mbox{for $i=0, 1$}.
  \end{equation}
  Indeed, the vanishing $\HH^0(\p(V),\cI_X (2))=0$
  (see our hypothesis \eqref{degree}) implies
  $\HH^0(\p(V),\cI_X \ts \Omega_{\p(V)}(2))=0$,
  %%% FIX è contenuto nell'altra ipotesi since $X$ is non degenerate
  while $\HH^1(\p(V),\cI_X (1))=0$ (see \cite[Proposition
  1]{bazan-mezzetti}) implies $\HH^1(\p(V),\cI_X \ts \Omega_{\p(V)}(2))=0$.

  We will use the equality \eqref{IXOMEGA} in order to prove \eqref{H1=1}.
  We have shown in Lemma \ref{O1} that the sheaf $\sL_\phi$ is the extension
  by zero of the line bundle $q_*(p^*(\OO_Y(1)))$ on $X$.
  So, restricting \eqref{seq:I} to $X$ we obtain an exact commutative
  diagram:
 \begin{equation}
   \label{diagramma-2}
    \xymatrix@-1ex{
      & 0 \ar[d] & 0 \ar[d] \\
       & U^* \ts \cI_{X} \ar^-{\cong}[r] \ar[d] & U^* \ts \cI_X \ar[d]  \\
      0 \ar[r] & I_\phi \ar[d] \ar^-{\tra \phi_{m}}[r] & U^* \ts \OO_{\p(V)} \ar[r] \ar[d] & \sL_\phi \ar[r] \ar^-{\cong}[d] & 0 \\
      0 \ar[r] & J \ar[r] \ar[d] & U^* \ts \OO_{X} \ar[r] \ar[d] & q_*(p^*(\OO_Y(1))) \ar[r]  & 0 \\
      & 0 & 0 
    }  
  \end{equation}
  where the sheaf $J$ is defined as the kernel of the map $U^* \ts
  \OO_{X} \to q_*(p^*(\OO_Y(1)))$, so that it is supported on $X$.
  We note that, applying $p_*q^*$ to the bottom row of the above
  diagram, we obtain the dual twisted Euler exact sequence restricted to $Y$, so that we have an
  identification:
  \begin{equation} \label{dettaglio}
    p_*(q^*(J)) \cong \Omega_{\p(U)|Y}(1).    
  \end{equation}
  Tensoring the leftmost column of \eqref{diagramma-2} by $\Omega_{\p(V)}(2)$ and taking
  cohomology, in view of \eqref{IXOMEGA} we find:
  \[
   \HH^0(\p(V),I_{\phi} \ts \Omega_{\p(V)}(2)) \cong \HH^0(X,J \ts \Omega_{\p(V)}(2)).
  \]
  So our final goal is now to show:
  \[
  \hh^0(X,J \ts \Omega_{\p(V)}(2)) = 1.
  \]
  
  In order to prove the above equality, 
  we tensor the lowest row of diagram \eqref{diagramma-2} by
  $\Omega_{\p(V)}(2)$ and
  we use the dual Euler sequence on $\p(V)$ twisted by
  $\OO_{\p(V)}(2)$.
  We get the following exact commutative diagram :
  \begin{equation}
    \label{diagramma-3}
    \xymatrix@-2.9ex{
      & 0 \ar[d] & 0 \ar[d] & 0 \ar[d]\\
      0 \ar[r] & J \ts \Omega_{\p(V)}(2)\ar[d] \ar[r]  & U^*\ts\Omega_{\p(V)}(2)\ts \OO_X \ar[r] \ar[d] & \sL_\phi \ts \Omega_{\p(V)}(2)\ar[r] \ar[d] & 0 \\
      0 \ar[r] & J \ts V^*\ts \OO_{\p(V)}(1)\ar[d] \ar[r]  & U^* \ts V^*\ts \OO_X(1) \ar[r] \ar[d] & \sL_\phi \ts V^*\ts \OO_{X}(1)\ar[r] \ar[d] & 0 \\
      0 \ar[r] & J\ts \OO_{\p(V)}(2) \ar[r] \ar[d] & U^* \ts \OO_{X}(2) \ar[r] \ar[d] & \sL_\phi\ts  \OO_{X}(2)\ar[r] \ar[d]  & 0 \\
      & 0 & 0 & 0
    }  
  \end{equation}
  Because the sheaf $J\ts\Omega_{\p(V)}(2)$ is supported on $X$,
  we have the natural isomorphisms:
  \begin{align*}
    \HH^0(\p(V),J\ts\Omega_{\p(V)}(2)) & \cong
    \HH^0(X,q^{*}(J \ts\Omega_{\p(V))}(2))) \cong \\
    & \cong \HH^0(Y,p_{*}q^{*}(J \ts \Omega_{\p(V)}(2))).
  \end{align*}
  Recall that our aim is to show that the above vector space is
  one-dimensional.
  To show this, in \eqref{diagramma-3} we pull-back by $q$ and
  push-forward by $p$.
  By the argument of Lemma \ref{projection}, we can write $p_*(q^*(\sL_\phi \ts \Omega_{\p(V)}(2)))$
  as $\im(\delta)$, where $\delta$ is defined by \eqref{symm}.
  Using \eqref{dettaglio}, we can thus write the following exact commutative diagram:
  \[
    \xymatrix@-2.5ex{
      & 0 \ar[d] & 0 \ar[d] & 0 \ar[d]\\
     0 \ar[r] & \Omega_{\p(U)|Y}\ts\im(\delta)\ar[d] \ar[r]  & U^*\ts\im(\delta)(-1) \ar[r] \ar[d] & \im(\delta)\ar[r] \ar^-{}[d] & 0 \\
      0 \ar[r] & \Omega_{\p(U)|Y} \ts V^*\ts \cE(1)\ar[d] \ar[r]  & U^* \ts V^*\ts \cE \ar[r] \ar[d] &  V^*\ts \cE(1)\ar[r] \ar^-{}[d] & 0 \\
      0 \ar[r] & \Omega_{\p(U)|Y}\ts \Sym^2 \cE(1) \ar[r] \ar[d] & U^* \ts  \Sym^2 \cE(1) \ar[r] \ar[d] &  \Sym^2 \cE(1)\ar[r] \ar[d]  & 0 \\
      & 0 & 0 & 0
    }  
  \]

  This gives:
  \[
  p_*(q^*(J \ts \Omega_{\p(V)}(2))) \cong \im(\delta) \ts \Omega_{\p(U)|Y}. 
  \]
  So we obtain the following isomorphism:
  \[
  \HH^0(X,J \ts \Omega_{\p(V)}(2)) \cong
  \HH^0(Y,\im(\delta) \ts \Omega_{\p(U)|Y}).
  \]
  We have thus reduced the problem to show:
  \[
  \hh^0(Y,\im(\delta) \ts \Omega_{\p(U)|Y}) = 1. 
  \]

  We use now Claim \ref{claim1} and we twist \eqref{Fcal} by $\Omega_{\p(U)}$.
  Taking cohomology we obtain $\HH^p(Y,\cF \ts \Omega_{\p(U)}(1))=0$ for
  $p=0,1$ since $\HH^p(Y, \Omega_{\p(U)}(-1)) = 0$ and $\HH^p(Y, \Omega_{\p(U)}(1)) = 0$.
  In view of the vanishing just obtained, tensoring \eqref{Fbar} by $\Omega_{\p(U)}$ and taking
  cohomology, we get, for $m\ge 4$:
  \[
  \hh^0(Y,\im(\delta) \ts \Omega_{\p(U)|Y}) = \hh^1(Y, \Omega_{\p(U)|Y}) = 1.
  \]
  %%%%%%%%%%%%%%%%%%%%%%%%%%%%%%%%%%%%%%
  %%%%%%%%%% DIMOSTRAZIONE CASO m=3 %%%%%%%%%%%%%%
  %%%%%%%%%%%%%%%%%%%%%%%%%%%%%%%%%%%%%%
In order to complete the proof of the lemma we need to take care of  the case $m=3$. 

Let us use the notation of the proof of Lemma \ref{vanishing}. For every $m$ we write the exact sequence \eqref{seq:K}
  \begin{eqnarray*}
     0 \to K_m \to \cT_{\p(V)}(-2) \to {I}_m \to 0.
  \end{eqnarray*}
The leftmost  exact sequence in the exact diagram \eqref{diagramma}  
 \begin{eqnarray*}
     0 \to \cI_{X_m} \to {I}_m \to {I}_{m-1} \to 0
  \end{eqnarray*}
  induces, using the snake lemma, the following exact sequence
   \begin{eqnarray*}
     0 \to K_m \to K_{m-1} \to {\cI}_{m} \to 0
  \end{eqnarray*}
  which for $m=4$ gives
   \begin{eqnarray}
   \label{K3}
     0 \to K_4 \to K_{3} \to {\cI}_{4} \to 0.
  \end{eqnarray}
  Tensoring  \eqref{K3} with $\Omega_{\p(V)}(2)$ and taking cohomology we see that $\HH^1(\p(V),K_3 \ts \Omega_{\p(V)}(2))=0$
 since we have already proved $\HH^1(\p(V),K_4 \ts \Omega_{\p(V)}(2))=0$
and 
  $\HH^1(\p(V),\cI_X \ts \Omega_{\p(V)}(2))=0$ by \eqref{IXOMEGA}.
\end{proof}

\begin{lemma} \label{lift}
  Let $m,k\geq 3$ be integers, $V$ and $U$ be vector spaces of
  dimension respectively $2k$ and $m$.
  Let $\phi_1,\phi_2: U \ts \OO_{\p(V)} \to \Omega_{\p(V)}(2)$ be two
  morphisms such that $X_1 = X_{\phi_1}$ and $X_2 = X_{\phi_2}$ are
  smooth. Assume that the hypothesis of Section \ref{hyp} are satisfied at least by the
morphism $\phi_2$. If the two sheaves $\sL_1 = \cok(\tra \phi_{1})$ and $\sL_2 =
  \cok(\tra \phi_{2})$ are isomorphic, then there is an invertible
  matrix $\Phi \in \GL(U^*)$ and a nonzero scalar $\lambda$ such that the following diagram commutes:
  \begin{equation} \label{coniugate}
    \xymatrix{
      \cT_{\p(V)}(-2) \ar^-{\tra \phi_1}[r] \ar@{=}^{\lambda \id}[d] &  U^* \ts \OO_{\p(V)} \ar[d]^-{\Phi} \\
      \cT_{\p(V)}(-2) \ar^-{\tra\phi_2}[r] &  U^* \ts \OO_{\p(V)} \\
    }
  \end{equation}
\end{lemma}
\begin{proof}
  Consider the composition $\sigma$ of the surjection $U^{*} \ts \OO_{\p(V)} \to
  \sL_{1}$ with the isomorphism $\sL_{1} \to \sL_{2}$.
  In order to lift $\sigma$ to a map $\Phi: U^{*} \ts \OO_{\p(V)} \to U^{*}
  \ts \OO_{\p(V)}$, we have to check the vanishing of the group:
  \[
  \Ext^{1}_{\p(V)}(\OO_{\p(V)},I_{\phi_{2}}) \cong \HH^{1}(\p(V),I_{\phi_{2}}).
  \]  
  Since this is guaranteed by Lemma \ref{vanishing}, we get the map
  $\Phi$.
  The kernel of $\Phi$ is isomorphic to a direct sum of copies of
  $\OO_{\p(V)}$. Moreover, since the map $\sL_{1} \to \sL_{2}$ is an
  isomorphism, $\ker(\Phi)$ fits as a sub-sheaf of $I_{\phi_{1}}$, and
  thus provides a nonzero element in
  $\Hom_{\p(V)}(\OO_{\p(V)},I_{\phi_{1}}) \cong
  \HH^{0}(\p(V),I_{\phi_{1}})$. But this group is zero by Lemma
  \ref{vanishing}, so $\Phi$ is an isomorphism.

  So we have an induced isomorphism $I_{\phi_{1}} \to I_{\phi_{2}}$ and we
  consider the composition $\sigma'$ of the projection $\cT_{\p(V)}(-2) \to I_{\phi_{1}}$ with this isomorphism.
  We would like to lift $\sigma'$ to an automorphism $\Phi'$ of
  $\cT_{\p(V)}(-2)$.
  So we have to check the vanishing of:
  \[
  \Ext^{1}_{\p(V)}(\cT_{\p(V)}(-2),K_{\phi_{2}}) \cong
  \HH^{1}(\p(V),K_{\phi_{2}} \ts \Omega_{\p(V)}(2)).
  \]
   This is provided by Lemma \ref{lemmone2}, so we have 
  $\Phi'$. Clearly the map $\Phi'$ is not zero, hence it is a (nonzero) multiple
  of the identity since $\cT_{\p(V)}$ is a simple sheaf.
\end{proof}
Up to
multiplying $\tra \phi_2$ by a nonzero scalar, we may assume that $\Phi'$
is in fact the identity.

\begin{proof}[Proof of the main theorem]
  Let us show that the map $\rho : \G(m,\wedge^{2} V^{*}) \to \cH_m(V)$ is
  generically injective.
  By contradiction, we take two maps $\phi_{1}$ and $\phi_{2}$
  satisfying the hypothesis of Section \ref{hyp}, and such that
  $X_{\phi_{1}}=X_{\phi_{2}}=X$.
  Then the scroll $X$ is smooth, and the maps $\phi_{1}$ and
  $\phi_{2}$ define two line bundles on $X$. Namely, the cokernel
  sheaves $\sL_{i} = \cok(\tra \phi_{i})$ satisfy $\sL_{i}\cong
  q_{*}(p^{*}(\sM_{i}))$ for some line bundles $\sM_{i}$ on the base
  $Y$ of the scroll $X$.
  But by Lemma \ref{O1}, both line bundles $\sM_{1}$ and $\sM_{2}$ are
  isomorphic to $\OO_{Y}(1)$. So we have an isomorphism $\sL_{1} \cong
  \sL_{2}$, and we can thus apply Lemma \ref{lift}.
  In view of diagram \eqref{coniugate}, the maps $\tra\phi_{1}$ and
  $\tra\phi_{2}$ are thus conjugate by an isomorphism $\Phi \in
  \GL(m,\kk)$. Therefore,  the maps $\phi_{1}$ and
  $\phi_{2}$ correspond to the same point in the Grassmann variety
  $\G(m,\wedge^{2} V^{*})$. This proves that $\rho$ is generically injective.

  To complete the proof for $m\geq 4$, recall that we have computed the dimension of $\HH^{0}(X,\cN)$ in Theorem
  \ref{thm:normal}. For $m\geq 4$, this equals the dimension of
  $\G(m,\wedge^{2} V^{*})$. This implies that the map $\rho$ is
  dominant onto its image, and in fact a birational map of
  $\G(m,\wedge^{2} V^{*})$ onto $\cH_m(V)$ for $m\geq 4$.

  Let us now discuss the case $m=3$. We have proved that $\rho$ is birational onto its image.
  Thus an open dense subset of the image of $\rho$ is
  immersed in the subscheme $\cP \subset \cH_3(V)$ whose general element consists of a
  general smooth plane curve $Y$ of degree $k$ equipped with a general
  stable rank-$2$ vector
  bundle $\cE$, with $c_1(\cE)=K_Y+2 H_Y$ (see \cite{beauville:determinantal}).

  In order to finish the proof, we check now that $\cP$ is a subscheme of $\cH_3(V)$ whose dimension equals
  $3(k(2k-1)-3)=3(2k^{2}-k-3)=\dim(\G(3,\wedge^{2} V^{*}))$.
  To do this, let $Y$ be a general projective curve of genus $g={k-1 \choose{2}}$, equipped
  with a general stable rank-$2$ bundle of degree $k(k-1)$ defined on $Y$.
  Then $\ppp(\cE)$ lies in $\cH_3(V)$ (see for instance \cite{calabri-ciliberto-flamini-miranda:non-special}).
  The condition that the curve is planar is of codimension
  $k^{2}-6k+8$ in $\cH_3(V)$.
  Moreover, the rank-$2$ bundle $\cE(-H_{Y})$ should
  have determinant equal to $K_{Y}$, that is it should lie on the
  fiber of $K_{Y}$ under the 
  determinant map $\det : \sM_{Y}(2,2g-2) \to \Pic(Y)$, where
  $\sM_{Y}(2,2g-2)$ denotes the moduli space of stable rank-$2$ bundles
  of degree $2g-2$ on $Y$.
  This is a condition of codimension $\dim(\Pic(Y))={k-1 \choose{2}}$, so $\cP$ has
  codimension $3/2k^{2}-15/2k+9$ in $\cH_3(V)$.
  Since $\dim(\cH_3(V)) = 3k(5k-7)/2$ by Theorem \ref{thm:normal}, this implies that
  $\cP$ has dimension $3(2k^{2}-k-3)$. 
  This completes the proof.
\end{proof}

Our theorem asserts that $\rho$ is generically injective onto the set of
Palatini scrolls. 
However, one could consider the following related question.
Namely, we consider the pfaffian map which associates to $\phi
\in W = U^* \ts \wedge^2 V^*$ the pfaffian of $M_\phi$ as an element of
$|\OO_{\p(U)}(k)|$.
This factors through a map $\overline{\Pf}$ defined on the quotient of
$W$ by $\GL(V)$ acting by congruence. Beauville proved in
\cite{beauville:determinantal} that $\overline{\Pf}$ is generically
injective for $m=6$, $k=3$.

\begin{problem}
Is the map $\overline{\Pf}$ generically injective for $m\geq 7$,
$k\geq 3$? For $m=6$, $k\geq
4$? For $m=5$, $k\geq 6$? For $m=4,k\geq 16$?
\end{problem}

\begin{acknowledgements}{We would like to thank Luca Chiantini, Ciro Ciliberto,  Flaminio Flamini and
  Emilia Mezzetti for many interesting discussions and for calling our
  attention to 
  some useful references.
We would also like to thank  the referee for helpful comments  and for pointing  out a mistake in our first version of the paper.}
\end{acknowledgements}
%\bibliographystyle{amsalpha-my}
%\bibliography{bibliography}

%\vspace{-1cm}

\end{document}